\documentclass[a4paper,12pt,leqno]{article}

\usepackage{geometry}
\usepackage{tikz-cd}\usepackage{bussproofs}
\usepackage{amssymb,amsmath}
\usepackage{proof}
\usepackage{psvectorian}
\usepackage{CJKutf8}
\usepackage[utf8]{inputenc}
\usepackage{graphicx}
\usepackage{graphicx}
\graphicspath{ {images/} }
\usepackage{qtree}
\usepackage{mathtools}
\usepackage{hyperref}
\usepackage{listings}
\usepackage{indentfirst}
\usepackage{titlesec}
\numberwithin{equation}{section}

\newtheorem{defn}[equation]{Definition}

\newtheorem{rem}[equation]{Remark}
\newtheorem{exm}[equation]{Example}
\newtheorem{lemma}[equation]{Lemma}

\newtheorem{notat}[equation]{Notation}
\newtheorem{newpar}[equation]{}

\newtheorem{xdefn}{Definition.}
\newtheorem{xproposition}{Proposition.}
\newtheorem{xcorollary}{Corollary.}
\newtheorem{xrem}{Remark.}
\newtheorem{xexm}{Example.}
\newtheorem{xlemma}{Lemma.}
\newtheorem{xtheorem}{Theorem.}
\newtheorem{xnotat}{Notation.}
\newtheorem{xnewpar}{\it}
\newtheorem{xproof}{{\it Proof. }}
\newtheorem{xproofof}{{\it Proof}}

\newenvironment{definition}{\begin{defn}\em}{\end{defn}}
\newenvironment{remark}{\begin{rem}\em}{\end{rem}}

\newenvironment{proof}{\begin{xproof}\em}{\end{xproof}}

\newenvironment{newparagraph*}[1]{\begin{xnewpar}\hspace*{-1.5mm}{#1}. \rm}{\end{xnewpar}}

\newenvironment{definition*}{\begin{xdefn}\em}{\end{xdefn}}
\newenvironment{remark*}{\begin{xrem}\em}{\end{xrem}}
\newenvironment{example*}{\begin{xexm}\em}{\end{xexm}}
\newenvironment{notation*}{\begin{xnotat}\em}{\end{xnotat}}
\newenvironment{proposition*}{\begin{xproposition}}{\end{xproposition}}
\newenvironment{corollary*}{\begin{xcorollary}}{\end{xcorollary}}
\newenvironment{lemma*}{\begin{xlemma}}{\end{xlemma}}
\newenvironment{theorem*}{\begin{xtheorem}}{\end{xtheorem}}

\titleformat*{\section}{\large\bfseries}

\begin{document}

\title{A Logic for Aristotle's Modal Syllogistic}

\author{Clarence Lewis Protin$^1$\footnote{The author declares that there is no conflict of interest.}}

\date{$^1$Independent Researcher, cprotin@sapo.pt\\[2ex]%
	\today}

\maketitle
\begin{abstract}
We propose a new  modal logic endowed with a simple deductive system  to interpret Aristotle's theory of the modal syllogism. While being inspired by standard propositional modal logic it is also a logic of terms that admits a (sound) extensional semantics involving possible states-of-affairs in a given world. Applied to the analysis of Aristotle's modal syllogistic as found in the \emph{Prior Analytics} A8-22 it sheds light on various fine-grained distinctions which when made allow us  to  clarify some ambiguities and obtain a completely consistent system and prove all of the modal syllogisms  considered valid by  Aristotle.This logic allows us also to make a  connection with the  axioms of modern propositional modal logic and to perceive to what extent these are implicit in Aristotle's reasoning.  Further work wil involve addressing  the question of the completeness of this logic (or variants thereof)  together with the extension of the logic to include a calculus of relations (for instance the relational syllogistic treated in Galen's \emph{Introduction to Logic}) which Slomkowsky has argued is already found in the \emph{Topics}.\\


\end{abstract}


\section{Introduction}

Aristotle's theory of the "modal" syllogism in the \emph{Prior Analytics} A8-22 has traditionally been considered one of the most problematic parts of the entire Aristotelic \emph{corpus}. Opinions as to the overall consistency and merit of this theory have generally been unfavourable (see for instance \cite{kne, rijen, Stri}, \cite{call}[p.1]).
The different notions of  "contingent", "possible" or "necessary"  predication seem to be inseparably tied to Aristotle's physics and ontology\cite{pat, call, thom} and furthermore do not seem to be used by Aristotle in an entirely consistent way in his proofs of the validity of the different syllogisms or his counter-examples. While much of Aristotle's theory seems to anticipate modern formal logic in a striking way (in particular his use of concrete models to disprove the validity of certain syllogisms and his use, as the modern mathematical logician,  of a metalogic\footnote{We use the term "metalogic" to refer to the reasoning processes Aristotle employs to deduce results about syllogisms.} to formally prove the validity of syllogisms) there is also much that  seems quite intractable to modern methods. Since the days of Bocheński  and   Łukasiewicz  and specially more recently there has been an active interest in interpreting the modal syllogistic using modern formal methods\cite{ rini, malink, bo, luka}. But these approaches end up being quite involved and tied up with semantic and ontological problems.

 Aristotle's concept of "term" would seem to correspond to either an individual or to a predicate (in general unary). Such terms have an \emph{extension}, the totality of individuals which fall under them. The extensionalist interpretation of the predication \emph{all $A$ is $B$} is that all individuals which fall under $A$ also fall under $B$: in modern predicate logic, $\forall x. A(x) \rightarrow B(x)$.  This interpretation
is that of the standard Tarskian semantics of classical first-order logic. It is an interpretation which finds its most adequate use in  \emph{mathematics}.  It is certain that Aristotle was highly influenced by the mathematics of his time in developing the theory in the \emph{Analytics}. But the goal of the \emph{Organon} was to develop a general methodology for valid reasoning in philosophy and science, to find a formal logic for \emph{natural language based} reasoning. It is not necessary to recall here the serious problems of extensionalism (at least in its naive form) once we leave the domain of mathematics. Frege already remarked that in \emph{All men are mortal}
we are not referring to the relationship between the unknown, vague and context-dependent extensions of the concepts of Man and Mortal but rather to a relationship of \emph{concepts}: our concept of $Man$ includes the concept of $Mortal$. It is clear that if we employ extensions then these must be partial, local, well-defined and reflect epistemic and spatio-temporal limitations.

We point out that the "logic" present in the \emph{Topics}, which probably reflects to a large extent the rules of formal debate prevalent in Greek philosophical schools since the time of Socrates,  is clearly much more  language-based than that of the \emph{Analytics} as well as having a richer ontological texture. The \emph{Topics} also contain implicitly  and in a more intensional form much of the reasoning in the \emph{Analytics} (cf. Malink argues for this in \cite{malink}[chs. 8-9]). The  syllogistic theory of the \emph{Analytics} might be said to hint at a greater importance of extension in the choice of terminology (major, minor or middle premise), in some of the arguments used to establish the validity of syllogisms and the importance of mathematical examples. As mentioned above, Aristotle's use of counter-examples in the treatment of the modal syllogism   anticipates modern modal logic. But what is clear is that in the \emph{Analytics} we have categorical and ontological neutrality (in modern terminology, terms are not "typed"):   the category of a term, or its being a genus, species, proper or accident, and the particular type of relationship between the two terms (species to genus, accident to substance, etc.)  is completely immaterial. As Striker put it in \cite{Stri}: \emph{that assertoric syllogistic is, as it were, metaphysically neutral from the point of view of Aristotle’s ontology}.  However Aristotle' became conscious of difficulties in this approach to the modal syllogism and invoked the elimination of considerations of time in predication (I Pr. An. 34b 6 -11).

In this work we consider  a type-neutral formal logic  in the style of the \emph{Analytics}. Although we will provide a sound extensional semantics for our logic (and be guided by the concrete intuition it provides), since the logic does not contain an equality predicate it can also be interpreted intentionally as a logic of concepts.  The  neutral position we take is that terms \emph{have} extensions but are not  to be identified with their extension and hence term equality should not be conflated with equality of the corresponding extensions.

 For reasons of both elegance and faithfulness to the structure of Aristotle's logic, we do not employ modern predicate logic as in \cite{luka, rini} for instance. Modern classical propositional logic (anticipated by the Stoics but also clearly present in portions of the \emph{Analytics}) is also clearly inadequate.  Our logic, called Aristotelic Modal Logic (AML),  is a fragment of minimal propositional logic built over a rudimentary modal term calculus. It is given a simple sound  semantics in which each model consists of a collection of possible state-of-affairs for a given world.  Connections can be made with algebraic logic, in particular Boolean algebras with operators (modal algebras). Our approach is in fact inspired by Boole's original method of formalising the (assertoric) syllogism.  As a term calculus AML incorporates term negation and both particular and universal predication as primitive notions as well as two forms of necessitation $\square$ and $\sqcup$  which give rise to two fundamental forms of possibility/contingency which correspond to
 two different term-operators $\lozenge$ and $\ominus$ that can be applied to a term.  Aristotle does appear to have made such a distinction as seen in his use of the terms   \emph{endekhomenon} and \emph{dunaton} . AML  also expresses several  fundamental axioms of S5 modal logic which we find were implicity used by Aristotle.   Different combinations of the above operators allow us to distinguish up to eight different kinds of contingent predication of one term to another.  We are also lead to distinguish between a weak and strong form of assertoric predication.  Many of the difficulties in Aristotle's theory can be traced to a confusion in certain instances between a certain type of "contingent", weak assertoric, strong assertoric and "necessary" predication.

  Equipped with these concepts we can make a fine analysis of the entirety of the theory of the syllogism in A4-A22 and explain the ambiguities and alleged errors and inconsistencies in Aristotle's theory by making the concepts precise and selecting adequate interpretations within AML so as to obtain a completely consistent interpretation with only minor alterations or clarifications.  
 
  In what follows we will make use of the Ross edition of the text of the Prior Analytics\cite{Ros}.

\section{Aristotelic Modal  Logic}

\begin{definition}

The terms and formulas of \emph{Aristotelic Modal Logic} (AML) over a countable set $T$ of atomic terms $t_1, t_2,...,t_n,...$ are  defined as follows:
\begin{itemize}

 \item A \emph{term} is either an atomic term  $t$ or an expression of the form
$A^c$, $\square A$ or $\sqcup A$ where $A$ and $B$ are terms.

\item Atomic formulas are  defined as follows: if $A$ and $B$ are terms then   $A\rightarrow B$ and $A \leadsto B$ are atomic formulas. The last two are called 
\emph{universal} and \emph{particular} formulas.
\item Formulas are either atomic formulas or expressions of the form $\phi\enspace \&\enspace \psi$ or $\phi\Rightarrow \psi$ where $\phi$ and $\psi$ are  formulas. 

\end{itemize}
\end{definition}



We use the notation $\lozenge A = (\square A^c)^c$ and $\ominus A = (\sqcup A)^c$. $\ominus$ represents bidirectional possibility which we call \emph{bicontingency} (see the following remark).
By analogy we call $\sqcup$ the \emph{binecessity} operator.
\begin{remark}
We could also introduce $\wedge$  as a primitive binary term operator in place of $\leadsto$ and $\sqcup$. Then we  could define $A\leadsto B$ as $\star(A\wedge B)$ and $\sqcup A$ as $(\lozenge A \wedge \lozenge A^c)^c$. But the presentation above is chosen to be more faithful to Aristotle's approach.
\end{remark}

We present  the deductive system of AML in natural deduction style.
We have the usual  rules for $\&$ and $\Rightarrow$ for minimal propositional logic:

\vspace{5mm}

\AxiomC{$\phi$ }\AxiomC{$\phi\Rightarrow \psi$ }\RightLabel{MP}\BinaryInfC{$\psi$}
\DisplayProof

\vspace{5mm}

\AxiomC{$[\phi]$ }
\noLine
\UnaryInfC{$\psi$ }\RightLabel{\enspace $\Rightarrow$I}
\UnaryInfC{$\phi \Rightarrow \psi$}
\DisplayProof 

\vspace{5mm}

\AxiomC{$\phi$ }\AxiomC{$ \psi$ }\RightLabel{\enspace $\&$I}
\BinaryInfC{$\phi \enspace \&\enspace \psi$}
\DisplayProof

\vspace{5mm}

\AxiomC{$\phi_1 \enspace \& \enspace \phi_2$ } 
\RightLabel{\enspace $\&$E$_i \enspace ( i = 1,2 )$}
\UnaryInfC{$\psi_i$}
\DisplayProof

\vspace{5mm}

Then we have rules for $(\enspace)^c$ and $\sqcup$:

\vspace{5mm}

\AxiomC{$A$ }\RightLabel{\enspace cEx}

\UnaryInfC{$A'$}
\DisplayProof

\vspace{5mm}
\emph{Where $A'$ results from $A$ by replacing any number of occurrences of $C^{cc}$ with
$C$ or $C^{cc}$ with $C$ for a subterms $C$ of $A$}.

\vspace{5mm}

\AxiomC{$A$ }\RightLabel{\enspace $\sqcup$Ex}

\UnaryInfC{$A'$}
\DisplayProof

\vspace{5mm}
\emph{Where $A'$ results from $A$ by replacing any number of occurrences of $\sqcup C^{c}$ with
	$\sqcup C$ or $\sqcup C$ with $\sqcup C^c$ for a subterm $C$ of $A$}.

\vspace{5mm}

\AxiomC{$A\rightarrow B$ }\AxiomC{$B\rightarrow C$ }\RightLabel{\enspace $\rightarrow$T}
\BinaryInfC{$A\rightarrow C$}
\DisplayProof

\vspace{5mm}


\AxiomC{$A\rightsquigarrow B$ }\RightLabel{\enspace$\leadsto$C}
\UnaryInfC{$B\rightsquigarrow A$}
\DisplayProof
\vspace{5mm}

\AxiomC{$A\rightsquigarrow B$ }\RightLabel{\enspace$\rightsquigarrow$I}
\UnaryInfC{$A\rightsquigarrow A$}
\DisplayProof
\vspace{5mm}

\AxiomC{$A\rightsquigarrow B$ }\AxiomC{$B\rightarrow C$ }\RightLabel{\enspace$\leadsto$T}

\BinaryInfC{$A\rightsquigarrow C$}
\DisplayProof

\vspace{5mm}
\AxiomC{$A\rightarrow  B^c$ }\RightLabel{\enspace cC}
\UnaryInfC{$B\rightarrow  A^c$}
\DisplayProof
\vspace{5mm}

\vspace{5mm}
\AxiomC{$A\rightarrow  B$ }\RightLabel{\enspace K}
\UnaryInfC{$\square A\rightarrow  \square B$}
\DisplayProof
\vspace{5mm}

	\AxiomC{$\square A\rightarrow B$ }\AxiomC{$\square A^c\rightarrow B$ }\RightLabel{\enspace $\sqcup$I}
\BinaryInfC{$\sqcup A\rightarrow B$}
\DisplayProof
\vspace{5mm}

Finally we have the axioms:

\[\square A \rightarrow \sqcup A \tag{S} \]
\[ \square A \rightarrow A \tag{T}\]
\[  \square A \rightarrow \square \square A \tag{$\mathbf{4}$}\]

We use the notation $\star A$ for $A\rightsquigarrow A$. $\star A$ is to be interpreted as expressing the fact that the term $A$ has non-empty extension.

\begin{lemma}
In AML we can derive:

\[\ominus A \rightarrow \lozenge A^c \tag{i}\]
\[ A \rightarrow \lozenge A \tag{ii}\]
\[\ominus A \rightarrow \lozenge A \tag{iii}\]
\[\square A^c \rightarrow \sqcup A \tag{iv}\]
\[   \lozenge \lozenge A \rightarrow \lozenge A \tag{v}\]
\[\lozenge\ominus A \rightarrow \ominus A \tag{vi}\]
\[  \lozenge \square A \rightarrow \lozenge A  \tag{vii} \]
We also get a derived rule $\ominus$Ex in which we can interchange  $\ominus A$ and $ \ominus A^c$ and the derived rules:\\

\AxiomC{$\star A$ }\RightLabel{\enspace$\star \lozenge$}
\UnaryInfC{$\star \lozenge A$}
\DisplayProof
\vspace{5mm}

\AxiomC{$\star \square A$ }\RightLabel{\enspace$\star\square $}
\UnaryInfC{$\star A$}
\DisplayProof
\vspace{5mm}

\AxiomC{$\star A$}
\AxiomC{$A\rightarrow B$ }\RightLabel{\enspace$\rightarrow$W}
\BinaryInfC{$A\rightsquigarrow B$}
\DisplayProof

\vspace{5mm}

	\AxiomC{$ A\rightarrow \lozenge B$ }\AxiomC{$A\rightarrow \lozenge B^c$ }\RightLabel{\enspace $\ominus$I}
\BinaryInfC{$A\rightarrow \ominus B$}
\DisplayProof

	\vspace{5mm}
\AxiomC{$A\rightarrow   B$ }\RightLabel{\enspace $\lozenge$T}
\UnaryInfC{$\lozenge A\rightarrow  \lozenge B$}
\DisplayProof
\vspace{5mm}

\vspace{5mm}
\AxiomC{$A\rightarrow  \lozenge B$ }
\UnaryInfC{$\lozenge A\rightarrow  \lozenge B$}
\DisplayProof
\vspace{5mm}

\AxiomC{$A\rightsquigarrow  B$ }\RightLabel{\enspace $\rightsquigarrow$K}
\UnaryInfC{$\lozenge A\rightsquigarrow  \lozenge B$}
\DisplayProof

\vspace{5mm}
\AxiomC{$A\rightarrow  \square B$ }
\UnaryInfC{$\square A\rightarrow  \square B$}
\DisplayProof
\vspace{5mm}

\end{lemma}

We note that the proof of (vi) makes essential use of rule $\sqcup$I.

The condition $\star A$ in $\rightarrow W$ (called the \emph{weakening} rule) reflects the classical problem that Aristotle's universal quantification "all $A$ is $B$" seems to implicitly assume the non-empty extension of $A$. From the  point of view of modern quanitifier logic  if we state that all unicorns
can recite the \emph{Iliad}  we cannot conclude therefrom that there exists a unicorn that can recite the \emph{Iliad}.
The logic of AML is minimal logic and as such has no negation. AML is not designed to capture the classical Aristolean theory of  contradictories although it would be interesting to considder an extension of the deductive system to accomplish this.  In any case note that  $\neg(\star A \enspace\&  \enspace A\rightarrow B)$ could not be $B \rightsquigarrow A^c$ for standard  negation $\neg$.

We denote by AML$^{S5}$ the extension of AML which results from adding  the analogue of the (S5) axiom:

\[ \lozenge A \rightarrow \square \lozenge A \tag{S5} \]

It does not however appear to be derivable from the current axioms and rules of AML. See the conclusion of this paper for a possible proof strategy.

We denote by $AML^{\blacklozenge}$ the extension of AML which includes
formulas of the form $\blacklozenge \phi$ with $\phi$ an atomic formula together with the rule : from $\phi$ we can derive $\blacklozenge \phi$. The significance of this extension will be explained further ahead.

Given two terms $A$ and $B$ of AML we use the notation $P(A,B)$ to denote
any formula of the form $XA'\bigcirc YB'$ where $\bigcirc$ is either $\rightarrow$ or $\rightsquigarrow$ and $X$ and $Y$ are either empty or one of the symbols $\square, \lozenge, \ominus, \sqcup$ and $A'$ is either $A$ or $A^c$ and $B'$ is either $B$ or $B^c$. We call such a formula a \emph{predication} of $B$ to $A$.

By necessary predication of $B$ to $A$ we mean $A \rightarrow \square B'$ where $B'$ is either $B$ or $B^c$.  We distinguish eight different kinds of \emph{contingent} predication of $B$ to $A$:\\

contingent1: $A \bigcirc \lozenge B'$

contingent2: $A\bigcirc \ominus B'$

contingent3: $\lozenge A \bigcirc \lozenge B'$

contingent4: $\ominus A \bigcirc \ominus B'$

contingent5: $\ominus A \bigcirc \lozenge B'$

contingent6: $\lozenge A \bigcirc \ominus B'$

contingent7: $\lozenge A \bigcirc B'$

contingent8: $\ominus A \bigcirc B'$\\

where $B'$ is either $B$ or $B^c$.
However only contingent1-6 will be of interest to us and of these chiefly contingent1-4.

An \emph{assertoric} predication $B$ to $A$ is $A \rightarrow B'$ where $B'$ is either $B$ or $B^c$.

Given $A$ and $B$ together with  symbols $X$ and $Y$  we have four distinguished types of predication $XA'\bigcirc YB'$:

A corresponds to $XA\rightarrow YB$

E corresponds to $XA\rightarrow YB^c$

I corresponds to $XA\leadsto YB$

O corresponds to $XA\leadsto YB^c$\\

A \emph{syllogism} in $A,B,C$ is a formula of the form

\[ P_1(X_1,X_2)\enspace \&\enspace P_2(X_3,X_4) \Rightarrow P_3(X_5,X_6) \]

in which the $X_i$ for $i=1,...,6$ are elements of $\{A,B,C\}$ and $X_1\neq X_2$, $X_3\neq X_4$ and $X_5\neq X_6$. The first two predications are called the \emph{premises} and the third is called the \emph{conclusion}

	
	
	
	

\begin{definition}
An \emph{extensional model} for AML is a pair $M = (I, V)$ consisting of a  non-empty set individuals $I$ and non-empty set $V$ of valuations $v: T \rightarrow \mathcal{P}(I)$.
\end{definition}

We define the \emph{extension} $E_v(A)$ of a  term $A$ for $v \in V$ inductively:\\

If $A$ is an atomic term $t$ then $E_v(t) = v(t)$

$E_v(A^c)$ = $I/ E_v(A)$

$E_v(\square A) = \bigcap_{v \in V} E_v(A)$

$E_v(\sqcup A) = \{a \in I: \forall v\in V. a\in E_v(A) \vee \forall v\in V. a\notin E_v(A)\}$\\

We define satisfaction of atomic formulas as follows:\\

$M \Vdash A \rightarrow B$ if $E_{v}(A) \subseteq E_{v}(B)$  for all $v \in V$.

$M \Vdash A \leadsto B$ if $E_{v}(A) \cap E_{v}(B) \neq \emptyset$  for all $v \in V$.

$M \Vdash \star A$ if $E_{v}(A)\neq \emptyset$ for all $v\in V$. \\

Satisfaction for complex formulas are defined as usual. 
A formula $\phi$ of AML is \emph{valid}  if it is satisfied in all extensional models.
If this is the case  we write $\Vdash \phi$.

Note that  it follows that $E_v(\lozenge A)  = \{x \in I: \exists w \in V. x \in E_w(A)\}$ and $E_v(\ominus A)  = \{x \in I: \exists u \in V. x \in E_u(A) \text{ and }   \exists w \in V. x \notin E_w(A)   \}$.

The elements of $V$ can be seen as different possible state-of-affairs for a world. Extensional models can be visualised nicely as modified Venn diagrams in which circles are given different colours according to the $v$ they express. Thus two circles may intersect for one colour and be disjoint for another. 

Our models also interpret AML$^\blacklozenge$: for an extensional model $M$ we have $M \Vdash \blacklozenge \phi$ according to the cases:
if $\phi$ is of the form $\star A$ then we interpretation is that same as that of $\star A$. If $\phi$ is $A\rightarrow B$ then the condition holds if there is a $v$ such that $E_v(A) \subset E_v(B)$. Similarly if $\phi$ is $A\rightsquigarrow B$ then the condition holds if there is a $v$ such that $E_v(A)\cap E_v(B)\neq \emptyset$.

The following can be easily checked:

\begin{lemma}
	The rules and axioms of AML, AML$^{S5}$ and AML$^\blacklozenge$ are sound for extensional models.
\end{lemma}

It is an open question if $S5$ is independent from AML and whether we have completeness for extensional models.

The following is easy (if somewhat tedious) to show:
\begin{lemma}
	There are no syllogisms in which the premises are particular predications.
\end{lemma}

\begin{remark}  Note that our models do not have, as usual,  a distinguished actual world
	$v_0 \in V$. The interpretation of $A\rightarrow B$ and $A\leadsto B$ is very strong, as suggested by Pr. An. 34b 6 -11.  It is in fact a hidden
	necessitation at the "predicate" level rather than the "term" level (cf. the rule $\rightarrow$K). A weaker interpretation of the assertoric in terms of an actual world would cause a great deal of problems. Our approach also saves us from having to deal with restrictions on the hypotheses in our natural deduction presentation of rule K. Rule K can be seen as the analogue of using  classical necessitation followed by K in the classical sense. One can perhaps think of our atomic formulas being prefixed by an invisible $\blacksquare A\rightarrow B$. 
	 In Aristotle there is in fact no great difference between syllogisms with only assertoric premises and those with only necessary premises.
	 The term \emph{anangkê} is also used at the metalogical level, in particular in the sense that the conclusion of a syllogism follows by necessity.
	 In the
	case of contingent predication it is quite arguable that Aristotle's concept was not
	a hypothetical predicate-level diamond $\blacklozenge A\rightarrow B$ but contingent predication as we have defined.  
	 Just consider the syntax of frequently used expressions
	of the type \emph{$A$ tô(i) de B panti endekhesthai}: $A$ may apply to all $B$, which should be interpreted for instance as $A\rightarrow \ominus B$. Also it is quite  remarkable that we find rule $\lozenge T$ directly expressed in Pr. An. 36a6-10.
\end{remark}

\section*{Preliminaries on Aristotle's  Theory of the Syllogism}

Let us consider syllogisms in AML with only assertoric predications. Assertoric predications  of $B$ to $A$ can be of type A,E,I or O. The classical
theory of the three figures of the assertoric syllogism is found in the \emph{Prior Analytics} A4-8 and is codified as

\begin{itemize}
	\item First figure: AAA, EAE, AII, EIO
	\item Second figure: EAE, AEE, EIO, AOO
	\item Third figure: *AAI, IAI, AII, *EAO, EIO, OAO
\end{itemize}

in which the first figure is of the form $P_1(B,A)\enspace\&\enspace P_2(C,B)\Rightarrow P_3(C,A)$, the second of the form  $P_1(B,A)\enspace\&\enspace P_2(C,A)\Rightarrow P_3(B,C)$ and the third
of the form  $P_1(C,A)\enspace\&\enspace P_2(C,B)\Rightarrow P_3(A,B)$.
We will allow  freedom on the order of the terms in the conlusion in the second and third figure when it is obvious that one conclusion is inter-derivable from another.  The star * means that we must add a third premise $\star C$ to the syllogism. This is so we can apply the weakening rule.

\begin{lemma} All the purely assertoric syllogisms in the three figures A4-A8 are derivable in AML.
\end{lemma}

\begin{proof}

As an example of a proof in AML consider Aristotle's derivation of the assertoric syllogism AAI of the third figure

\vspace{5mm}

\AxiomC{$\star A$}\AxiomC{$C\rightarrow  A$ }\AxiomC{$C\rightarrow B$ }\TrinaryInfC{$A\rightsquigarrow B$}
\DisplayProof

\vspace{5mm}

It could be expressed:

\vspace{5mm}
\AxiomC{$\star C$}
\AxiomC{$C\rightarrow  A$ }
\BinaryInfC{$C\rightsquigarrow A$}
\UnaryInfC{$A\rightsquigarrow C$}
\AxiomC{$C\rightarrow  B$}\BinaryInfC{$A\rightsquigarrow B$}
\DisplayProof
\vspace{5mm}

The syllogism EAO of the third figure could be proven:

\vspace{5mm}
\AxiomC{$C\rightarrow B$}
\AxiomC{$\star C$}
\AxiomC{$C\rightarrow A^c$}
\BinaryInfC{$C \rightsquigarrow A^c$}
\UnaryInfC{$A^c \rightsquigarrow C$}
\BinaryInfC{$A^c\rightsquigarrow B$}
\UnaryInfC{$B\rightsquigarrow A^c$}
\DisplayProof
	\vspace{5mm}

The syllogism OAO in the third figure is obtained

	\vspace{5mm}
	\AxiomC{$C\rightarrow B$}

	\AxiomC{$C \rightsquigarrow A^c$}
	\UnaryInfC{$A^c \rightsquigarrow C$}
	\BinaryInfC{$A^c\rightsquigarrow B$}
	\UnaryInfC{$B\rightsquigarrow A^c$}
	\DisplayProof
		\vspace{5mm}
		
and that of AOO in the second figure

\vspace{5mm}
\AxiomC{$C\rightsquigarrow A^c$}

\AxiomC{$B \rightarrow A$}
\UnaryInfC{$A^c \rightarrow B^c$}
\BinaryInfC{$C\rightsquigarrow B^c$}

\DisplayProof

\end{proof}

A well-known statement in the \emph{Prior Analytics}   36 b 35-37 a 9 is that universal contingent negatives do not convert. If we interpret "contingent" 
as contingent2 then this is true in AML:

\begin{lemma}
$\nVdash A \rightarrow \ominus B^c \Rightarrow B \rightarrow \ominus A^c$
\end{lemma}

\begin{proof}
We can use Aristotle's own example from 36 a 4-10.
Consider a model $M =(I,V)$ with $I = \{m, s\}$ and a AML language with atomic
terms $Man$, $Snow$ and $White$.  V consists of two valuations $v_0$ and $v_1$ with

$v_0(Man) = v_1(Man) = \{m\}$

$v_1(Snow) = v_1(Snow) = \{s\}$

$v_0(White) = \{s\}$

$v_1(White) = \{s,m\}$

That is to say, we have world with two individuals $m$ and $s$ with two possible state of affairs $v_0$ and $v_1$.  The man $m$ is always a $Man$ and
the portion of snow $s$ is always $Snow$. However there is a situation in which 
$m$ is $White$ and a situation in which $m$ is not $White$. $s$ is always $White$. We can check that

\[M\Vdash Man \rightarrow \ominus White^c\]

However we do not have

\[M\nVdash White \rightarrow \ominus Man^c\]

for clearly $s \in E_{v_0}(White)$ and $s$ cannot be both $Man$ and not a $Man$. $s$ is always not a $Man$ and hence $s \notin E_{v_0}(\ominus Man^c)$
\end{proof}

In a famous passage Aristotle argues for the conversion of the contingent particular (and this is used in the proofs of the modal syllogisms in the  third figure). For contingent3 and contingent4 this is immediate.  We note the following fact concerning contingent1 and contingent 2:

\begin{lemma}
 $\nVdash A\leadsto \lozenge B \Rightarrow B \leadsto \lozenge A$ and $\nVdash A\leadsto \ominus B \Rightarrow B \leadsto \ominus A$
\end{lemma}

Aristotle's arguments sometimes hinge on  the conversion of the necessary negative universal.  Also arguments which make this move have also  been argued by certain scholars to be fallacious (cf.\cite{loeb}[301]).

The interpretation of necessary predication we shall be using in the analysis of the modal syllogism is $A\bigcirc \square B$. It is remarkable that we have the following:

\begin{lemma}
	In AML$^{S5}$ we have
	
	\[C \rightarrow \square B^c \Rightarrow B \rightarrow \square C^c\]
\end{lemma}

\begin{proof}

Assume $C \rightarrow \square B^c$. Then it is easy to show that
by definition of $\lozenge$ we have $\lozenge B \rightarrow C^c$. Apply K
to obtain $\square \lozenge B \rightarrow \square C^c$. Since we have $B \rightarrow \lozenge B$ by lemma 2.3 (ii) the result follows from the S5 axiom.	
	
	\end{proof}

This solves the problem in \cite{call}[p.20]

\section*{First Figure}

A syllogism is of the \emph{first figure} if it is a formula of the form 

\[ P_1(B,A)  \enspace \& \enspace P_2(C,B) \Rightarrow  P_3(C,A) \]

We call the first atomic formula the \emph{major} premise and the second
formula the \emph{minor} premise. 

The following is straightforward:

\begin{lemma} There is no syllogism in the first figure in which the major premise has a particular arrow.
\end{lemma}

Given that the major must be universal, it is clear that the proof of a first figure syllogism  must use $\rightarrow$T.  The conclusion will be particular or universal accordingly as the minor is particular or universal. 

We now proceed to give a detailed analysis of Aristotle's treatment of the first figure in the \emph{Prior Analytics}. For this and the remaining figures we make use of Ross's commentary\cite{Ros}[pp.287-369] and notation for syllogism. Patterson\cite{pat} contains a nice table containing all the syllogisms of the three figures. We also follow Ross in distinguishing between contingent  and problematic premises which are denoted by superscripts $c$ (not to be confused with the $c$-operation in AML) and $p$ respectively. In what follows and in the next two sections we will
interpret "apodeictic" as right necessary $A \rightarrow \square B$ but will consider different interpretations of "contingent" in AML, in our terminology contingent1-6.

A8 deals with the case of two apodeictic premises. It is stated that the assertoric syllogisms remain valid if the premises are replaced by their corresponding apodeictic versions.  Consider the validities

\[ B \rightarrow \square A\enspace\&\enspace C\bigcirc \square B\Rightarrow C\bigcirc \square A  \tag{A1} \]

where $\bigcirc$ represents either a universal or particular arrow. Also $A$ may be replaced by $A^c$. Then it is easy
to see that we obtain all the first figure syllogisms:

 $A^nA^nA^n,E^nA^nE^n,A^nI^nI^n,E^nO^nO^n$.

A9 deals with the case in which one premise is apodeictic and another assertoric.

Here we find the most important instance in which Aristotle's  conclusions differs from AML. It seems that Aristotle's use of "assertoric" hovers between a "strong" assertoric (the invisble $\blacksquare (A \rightarrow B)$, in which the inclusion is valid in all states-of-affairs ) and a "weak" notion $\blacklozenge (A \rightarrow B$), in which the inclusion is valid one in one state-of-affairs, for instance, at one time or place, and even apodeictic in the sense of $A\rightarrow \square B$.  Thus in \cite{loeb}[p.190] it is shown that the relation between "man" and "animal" is used by Aristotle as an instance of both an "assertoric" and "apodeictic" predication thus demonstrating that the reading of "assertoric" as "strong assertoric" is legitimate.   That the weak notion is inadequate is shown by the fact that we cannot even derive $A^{\blacklozenge}A^{\blacklozenge}A^{\blacklozenge}$ from it. If there is a situation in which all dogs are sleeping and a situation in which all sleeping things are cats it does not follow that there is a situation in which all dogs are cats (this is related to the problem of "joint possibility" \cite{call}[p.88]).  We can however obtain syllogisms by mixing storng and weak assertoric premises.  For another observation regarding AML and "the problem of the two Barbaras" see the concluding section.

In A9 the following syllogisms are considered valid: $A^n A A^n, E^n A E^n, A^n I I^n, E^n I O^n$. These follow from the validities:

\[ B \rightarrow \square A\enspace\&\enspace C\bigcirc  B\Rightarrow C\bigcirc \square A  \tag{A2} \]

Aristotle argues that the following are invalid: $AA^nA^n, EA^nE^n, AI^nI^n, EI^nO^n$.

These are however syllogisms in AML for we have:

\[  B\rightarrow A \enspace \& \enspace C \rightarrow  \square B \Rightarrow C \rightarrow  \square A \tag{I}\]

which follows from applying K to the first premise.

To show that $AA^nA^n$ is not valid Aristotle produces a counter-example.   In this counter-example the "assertoric" major is not assertoric in the strong sense  but in the weak sense. Indeed the counter-example involves the terms "moving" (K), "animal"(Z) and "man"(A). We have indeed $A\rightarrow \square Z$  but $Z \rightarrow K$ cannot be considered as assertoric in the strong sense for there are  state-of-affairs in nature in which not all animals are in motion. Thus $A^{\blacklozenge} A^nA^n$,  which corresponds to 
\[  B\rightarrow \square A \enspace \& \enspace \blacklozenge (C \rightarrow B) \Rightarrow C \rightarrow  \square A \]

is invalid in AML. This can be easily shown by an extensional model counter-example.  In the same way we can show that $E^{\blacklozenge}A^nE^n$, $A^{\blacklozenge} I^nI^n$ and $E^{\blacklozenge}I^nO^n$ are invalid\footnote{We note that if we discard rule K then we do indeed solve the "problem of the two Barbaras" for then we have A2 but not I. However we cannot derive A1.}.

A14 deals with the case in which both premises are contingent.  The following syllogisms are argued to be valid:  $A^cA^cA^c, E^cA^cE^c, A^cE^cA^c, E^cE^cA^c, A^cI^cI^c, E^cI^cO^c, A^cO^cI^c$. The proofs employ transition from $E^c$ to $A^c$ and $O^c$ to $I^c$. 

We have the following validities in AML:

\[B \rightarrow \ominus A\enspace\&\enspace C \bigcirc \ominus B \Rightarrow C\bigcirc \ominus A \tag{A3}\]

If we interpret "contingent" as contingent2 then the validity of all the  above syllogisms follow from A3 (in which we may need to employ the Ex$\ominus$ rule).   If we use contingent4 then the same follows immediately. We can also validate those syllogisms which do not emply transition in the proof by interpreting "contingent" as contingent1 and using the validity

\[B \rightarrow \lozenge A\enspace\&\enspace C \bigcirc \lozenge B \Rightarrow C\bigcirc \lozenge A \tag{A4}\]

In A14 Aristotle then shows that there is no syllogism with contingent premises in which the major is particular. This is clearly true in AML no matter which interpretation of "contingent" we take.

In A15 Aristotle deals with the case in which one premise is "contingent" and the other "assertoric". He establishes the following cases as valid:

$ A^cAA^c, E^cAE^c,  AA^cA^p,  EA^cE^p,  EE^cE^p,  A^cII^c,  E^cIO^c,
AI^cI^p,  EI^cO^p,  AO^cI^p,  EO^cO^p$.

Let us take  again "contingent" as contingent2 and "problematic" as contingent1. This is justified because in \cite{loeb}[p.275] Aristotle is aware that the conclusion is not contingent2 (\emph{endekhomenon}) but simply $C$ not applying necessarily to $A$, that is, the weaker notion of contingency1 expressed by $A \rightarrow \lozenge B^c$ in AML. We have the following
validities:

\[B \rightarrow  A\enspace\&\enspace C \bigcirc \ominus B \Rightarrow C\bigcirc \lozenge A \tag{A5}\]
\[B \rightarrow  \ominus A\enspace\&\enspace C \bigcirc B \Rightarrow C\bigcirc \ominus A \tag{A6}\]

and it is easy to see that we can validate all the syllogisms above using A5 and A6.

Although Aristotle announces rule $\lozenge$T in this chapter\footnote{see remark 2.7} (which  is necessary in AML to derive the syllogisms we are considering) the proof Aristotle gives, which proceeds by \emph{ad impossibile}, has been found to be insufficient(cf \cite{loeb}[270-271]).

In A15 Aristotle also states that $A^cE$ and $E^cE$ do not yield a syllogism which is immediately verified using extensional model counter-examples. It is also clear that we do not obtain anything from two particular premises.

A16 deals with the case in which one premise is apodeictic and the other contingent. Aristotle begins by stating that all the cases of A15 yield syllogisms if "assertoric" is replaced by "apodeictic".   Let us consider the cases when both premises are universal. Aristotle considers
the following valid:  $A^nA^cA^p,  A^cA^nA^c, E^nA^c E( E^nA^cE^p), E^cA^nE^c, A^nE^cA^p$.

In AML we have the validities:

\[B \rightarrow  \square A\enspace\&\enspace C \bigcirc \ominus B \Rightarrow C\bigcirc \lozenge A \tag{A7}\]
\[B \rightarrow  \ominus A\enspace\&\enspace C \bigcirc \square B \Rightarrow C\bigcirc \ominus A \tag{A8}\]

Interpreting "contingent" and "problematic" as in A15 we validate all the syllogisms in the list except $E^nA^cE$.  It appears that we can only derive
this syllogism in AML$^{S5}$ using $\lozenge \square A \rightarrow \square A$.
 It is interesting that Aristotle in 36a15  seemingly endorses $A \rightarrow \lozenge A$ from lemma 2.3 (for the present interpretation of "problematic"). 
Aristotle's statement that $A^cE^n$ and $E^cE^n$ yield no syllogism can be easily verified using extensional models.

In A16 Aristotle  gives the following list of valid syllogisms in which one premise is particular: 
$E^nI^cO, 	 E^cI^nO^c,	 A^nI^cI^p$.  The last two follow from A8 and A7 and the first is validated in AML$^{S5}$ as for $E^nA^cE$.

\section*{Second Figure}

The second figure is a syllogism of the form

\[ P_1(B, A) \enspace\&\enspace P_2(C,A) \Rightarrow P_3(B,C)  \]

In proving the purely assertoric form of the second figure   in AML rule cC plays a key role. In AML we can generalise this rule in the form of the following easily provable equivalences:\\

c1. $ C \rightarrow \ominus B^c \Leftrightarrow \sqcup B^c  \rightarrow  C^c $

c2. $ C \rightarrow \lozenge B^c \Leftrightarrow \square B \rightarrow  C^c $

c3.   $ C \rightarrow \square B^c \Leftrightarrow  \lozenge B \rightarrow C^c $

c4.   $ C \rightarrow \sqcup B^c \Leftrightarrow  \ominus B^c \rightarrow C^c $\\

In A8 Aristotle deals with the case of two apodeictic premises. We must check the claim  that all the purely assertoric syllogisms of the second figure remain valid if we replace each premise and the conclusion by its corresponding apodeictic version. This follows the validities

\[  B\rightarrow  \square A \enspace \& \enspace C \rightarrow \square A^c \Rightarrow B \rightarrow \square C^c  \tag{B1}\]

proven by using c3 on the second premise and

\[  B\rightarrow  \square A^c \enspace \& \enspace C \rightarrow \square A \Rightarrow B \rightarrow \square C^c  \tag{B2}\]

which is proven by weakening the second premise to $C\rightarrow A$ and applying cC and then  K. 

A10 deals with the case of one assertoric and one apodeictic premise.  The following
are considered valid:  $E^nAE^n,	 AE^nE^n,	E^nIO^n$.  These follow from the validities

\[  B\rightarrow  \square A^c \enspace \& \enspace C \rightarrow  A \Rightarrow B \rightarrow \square C^c  \tag{B3}\]

\[  B\rightarrow   A \enspace \& \enspace C \rightarrow \square A^c \Rightarrow B \rightarrow \square C^c  \tag{B4}\]
and
\[  B\rightarrow  \square A^c \enspace \& \enspace C \rightsquigarrow  A \Rightarrow C \rightarrow \square B^c  \tag{B3}\]

In order to obtain the apodeictic conclusion in B3 (rather than a purely assertoric one) we have used the S5 axiom $\lozenge A \rightarrow \square \lozenge A$.

These syllogisms are invalid if we take "assertoric" to be weak assertoric. For instance for the first syllogism take B as 'sleeping', A as 'man' and C as 'horse'. Then we have $C\rightarrow \square A^c$ and $\blacklozenge (B\rightarrow A)$ for there can be a situation in which the men are asleep and the horses are awake. But clearly we do not have $B\rightarrow \square C^c$ because  there can be a situation in which a horse is sleeping.  We do however get $\blacklozenge (B \rightarrow \square C^c)$.

The following are considered invalid: $A^nEE^n, A^nOO^n, AO^nO^n$

In AML these are in fact valid. However if we take the assertoric premises to be weak assertoric (as is the case in Aristotle's own counter-example) their invalidity can be easily shown by concrete extensional models.

A17 treats of the case of two contingent premises. It is stated that there are no syllogisms wherein both premises are contingent (36b 25-30).  This is related to the fact that we cannot convert a universal contingent negative (at least for contingent1 and contingent2) but only have results like c1 and c2.

Let us consider the contingent2 premises $B\rightarrow \ominus A\enspace\&\enspace C\rightarrow \ominus A$. We can show that
we do not obtain a conclusion $B \rightarrow \ominus C$ by the following model
based on Aristotle's own example (37a 1-10). Let $M =(I,V)$ with
$I = \{m,h\}$ and consider a AML language with atomic terms $Man$, $Horse$ and $White$ and let $V$ consist of $\{v_1,v_2,v_3,v_4\}$ in which
$v_i(Horse) = \{h\}$ and $v_i(Man) = \{m\}$ for $i=1,...,4$ and such that
for $v_1(White) = \{m,h\}, v_2(White) = \{m\}, v_3(White) = \{h\}, v_4(White)=\emptyset$. Here we are considering that the individual man and horse $m$ and $h$ can in fact change their colour in different situations.

Then we have that $M\Vdash Man\rightarrow \ominus White\enspace\&\enspace Horse\rightarrow \ominus White$ but  $M\nVdash Man \rightarrow \ominus Horse$.

A18 treats of the case of one problematic and assertoric premise. The following are considered valid:$	EA^cE^p, A^cEE^p, EE^cE^p, E^cEE^p, EI^cO^p, EO^cO^p$. For contingent1 and contingent2 premises we have the following validities:

\[  B\bigcirc \lozenge A \enspace \& \enspace  C \rightarrow A^c \Rightarrow B \bigcirc \lozenge C^c  \tag{D1}\]

\[ B\bigcirc \ominus A \enspace \& \enspace  C \rightarrow A^c \Rightarrow B \bigcirc \lozenge C^c \tag{D2} \]

where $\bigcirc$ is to be replaced either by a universal or particular arrow.
These validate $A^cEE^p$. $EA^cE^p$ and $EI^cO^p$ are validated by

\[ B \rightarrow A^c\enspace\&\enspace C \bigcirc \lozenge A\Rightarrow C \bigcirc \lozenge B^c \tag{D4} \]

D4 for a contingent2 premise validate s$EE^cE^p$ snf $EO^cO^p$  by
transition (rule Ex$\ominus$). In the same way D2 validates $E^cEE^p$.

 Aristotle considers that $A^cA, AA^c, E^cA$ and $AE^c$ do not yield syllogisms. This can be checked in AML using extensional  model counter-examples.

A19 deals with the case of one apodeictic and one contingent premise. Valid syllogisms
are $ E^nA^cE^p,  E^nA^cE,  A^cE^nE^p, A^cE^nE,  E^nE^cE,  E^nE^cE^p,  E^cE^nE,  E^cE^nE^p$.

We have the following validity in AML:

\[  B\rightarrow \ominus A \enspace \& \enspace  C \rightarrow \square A^c  \Rightarrow B \rightarrow  C^c \tag{E1} \]

This validates $A^cE^nE, E^cE^nE$ and $A^cE^nE^p, E^cE^nE^p$ using $C^c \rightarrow \lozenge C^c$ to obtain a contingent1 conclusion. The other syllogisms can be validated in exactly the same way  by permuting the premises and renaming terms in E1. More generally we have

\[   B \rightarrow \square A^c  \enspace \& \enspace C\bigcirc \ominus A  \Rightarrow C \bigcirc  B^c \tag{E2} \]

Now Aristotle considers that the following premises do not yield a syllogism: 

$A^nE^c, E^cA^n, A^nA^c, A^cA^n$. 

But we can obtain syllogisms from such premises in AML.  The reason for this discrepancy is that in the examples Aristotle adduces  the "apodeictic" premise is not in fact right necessary (the interpretation we have been using) but merely (strong assertoric): it is necessary that everything awake(A) be in movement(M).
This evidently cannot be read as $A \rightarrow \square M$ but merely as strong assertoric $A \rightarrow M$.  

For particular premises Aristotle considers the following as valid:

 $E^nI^cO, E^nI^cO^p, E^nO^cO, E^nO^cO^p$.  
 
 The last two syllogism clearly follow from the first two by transition. The first two syllogisms follow from E2.  The invalid cases  considered by Aristotle are obtained by weakening
 the invalid cases discussed above for two universal premises.

\section*{Third Figure}

The third figure is a formula of the form

\[P_1(C,A)  \enspace \& \enspace P_2(C,B) \Rightarrow P_3(A,B)  \]

It appears that Aristotle is lax about the order of the premises
and thus considers in the third figure also formulas of the form

\[P_1(C,A)  \enspace \& \enspace P_2(C,B) \Rightarrow P_3(B,A)  \]

This follows immediately if $P_3(B,A)$ can be exchanged with $P_3(A,B)$.
If both figures are universal it is crucial to be able to pass to a particular form of one of the premises. This is in fact Aristotle's main strategy.  Hence in this situation we assume tacitly that we have an extra premise $\star C$ in all syllogisms of the third figure with two universal premises.  This is so we can make use of the weaking rule.
What formalisation of contingency should we use in this figure ? Aristotle makes frequent use of the conversion of the contingent particular as well as transition.
Hence an obvious choice would be to use contingent4 ($\ominus \ominus$).  This allows us to follow Aristotle's own proofs very closely. However it turns out that we can validate all the syllogisms in the third figure for other interpretations of contingency, in particular allowing different interpretations of contingency in the premise(s) and in the conclusion. We also follow Ross in considering the distinction between contingent (c) and problematic (p) as legitimately corresponding to two different interpretations of contingent.
 This is interesting because this shows for instances that we get valid syllogisms with contingent conclusion with only contingency1 in a premise. 
We cannot examine here the syllogisms for all the possible combinations of the six notions of contingency in the premises.
In the notation such as $AA^cI^p$  the first $A$ refers to the first premise and $A^c$ to the second.

In A8 we have the case of two apodeictic premises in the third figure. For our usual  interpretation of "apodeictic" these syllogisms are validated by the following syllogisms:

 \[  C\rightarrow \square A \enspace \& \enspace  C \rightarrow \square B \Rightarrow  A \rightsquigarrow \square B \]
 
  \[  C\rightsquigarrow \square A \enspace \& \enspace  C \rightarrow \square B \Rightarrow  A \rightsquigarrow \square B \]
 
Using these it is easy to justify Aristotle's conclusion that if we add "it is necessary" to the premises of all the purely assertoric syllogisms of the third figure we obtain valid syllogisms with apodeitic conclusions.

A11 deals with the case of one contingent and one apodeictic premise. 
For two universal premises valid syllogisms are $A^nAI^n, AA^nI$ and $E^nAO^n$. The first and the third are validated by

\[ C \rightarrow \square A\enspace\&\enspace C\bigcirc B\Rightarrow B\bigcirc \square A \]

To show $AA^nI$ we use

\[ C \rightarrow  A\enspace\&\enspace C\rightarrow \square B\Rightarrow A\rightsquigarrow  \square B \]

It also follows immediately that $A^nII^n$, $IA^nI$ and $E^nIO^n$ are valid.

Aristotle's argues that $EA^nO^n$ is invalid. In fact from such premises
we can only derive $A^c \rightsquigarrow \square B$ or an assertoric conclusion $B\rightsquigarrow A^c$.  That we do not get $B \rightsquigarrow \square A^c$ can be shown by a counter-example.

Consider a language with three atomic terms $A,B,C$ and a model $M = (I,V)$ with $I= \{x,y\}, V = \{v_0,v_1\}$ with
$v_0(A) =  v_1(C) = \{x\}$ and $v_0(C) = v_1(A) = \{y\}$ and let $v_i(B)  = I$ for $i=0,1$.  Then clearly $M \Vdash C \rightarrow A^c$ and $M\Vdash C \rightarrow \square B$. But clearly we $M \nVdash B \rightsquigarrow \square A^c$ since $v_i(\square A^c) = \emptyset$ for $i=0,1$. The same goes
for the invalidity of the conclusion $A \rightsquigarrow \square B^c$.
The same example can be used to show the invalidity of $OA^nO^n$ and $EI^nO^n$.

The other syllogisms with particular premises considered invalid by Aristotle are:

 $AI^nI^n, I^nAI^n,O^nAO^n$.
 
  In AML  however can derive these syllogisms. But consider Aristotle's example
  to invalidate $I^nAI^n$: A is "waking", B is "biped" and C is "animal". Certainly
  we have $C\rightsquigarrow \square B$ but we certainly do not have $C\rightarrow A$ in the AML sense of "assertoric".  Rather we have weak assertoric predication $\blacklozenge(C\rightarrow A)$. The text in fact reads:
  \emph{to de A tô(i) $\Gamma$ endekhetai} (31b30). Indeed if we take $I^nAI^n$
  as $I^nA^cI^n$ then this can be shown to invalid in AML using contingency1 or contingency2 for instance.


A20 deals with the case in which both premises are contingent. Aristotle establishes as valid:
$A^cA^cI^c, E^cA^cO^c , E^cE^cI^c, I^cA^cI^c, E^cI^cO^c, O^cA^cI^c, E^cO^cI^c, O^cE^cI^c$.

 If 'c' in the premises is interpreted as contingent1 then we obtain $A^cA^cI^c$ easily

\[C \rightarrow \lozenge A\enspace\&\enspace C\rightarrow \lozenge B \Rightarrow \lozenge A \rightsquigarrow \lozenge B \tag{C1} \]

with contingent3 conclusion and similarly we can validate $E^cA^cO^c$ but with $\lozenge B \rightsquigarrow \lozenge A^c$ as a conclusion.   In order to get $E^cE^cI^c$
we clearly must have a contingent2 or contingent4 for premise ofr instance. For example
we have

\[C \rightarrow \lozenge A^c\enspace\&\enspace C\rightarrow \ominus B^c \Rightarrow \lozenge B \rightsquigarrow \lozenge A^c \tag{C2} \]

In the proofs of C1 and C2 and similar validities we apply weakening to one of the premises (it does not matter which one). Hence it follows immediately that we validate the syllogisms which result from weakening one of the premises in
$A^cA^cI^c, E^cA^cO^c , E^cE^cI^c$ and hence we can validate in an analogous way
$I^cA^cI^c, E^cI^cO^c, O^cA^cI^c, E^cO^cI^c, O^cE^cI^c$.

A21 deals with the case in which one premise is contingent and another assertoric. According to Aristotle the valid syllogisms with two universal premises are:

$AA^cI^p, 	 A^cAI^c, EA^cO^p, E^cAO^c, AE^cI^p, EE^cO^p$

Consider the first syllogism $AA^cI^p$.

Let as take $c$ as contingent4. Then in AML we can prove a contingent5 conclusion

\[ C \rightarrow A\enspace\&\enspace \ominus C \rightarrow \ominus B \Rightarrow \lozenge A \rightsquigarrow \ominus B \tag{S1} \]

and hence also a contingent4 conclusion

\[ C \rightarrow A\enspace\&\enspace \ominus C \rightarrow \ominus B \Rightarrow \lozenge A \rightsquigarrow \lozenge B \tag{S2}\]

This is interesting because is follows Aristotle's own proof by converting the second premise. However we can also obtain the same conclusions with weaker contingent1 and contingent2 premises such as 

\[ C \rightarrow A\enspace\&\enspace  C \rightarrow \lozenge B \Rightarrow  A \rightsquigarrow \lozenge B \tag{S3} \]

The second syllogism $A^cAI^c$ can be treated in the same way. It is not clear why Ross writes 'c' here but 'p' in the conclusion of the former syllogism.
If in S1 and S2 and S3 we replace $A$ by $A^c$ then we get the same validations for $EA^cO^c$.  According to Aristotle $E^cAO^c, AE^cI^p, EE^cO^p$ follow from the first three syllogisms by conversion.  For $E^cAO^c$ we have
the valid syllogisms

\[ \ominus C \rightarrow  \ominus A^c\enspace\&\enspace  C \rightarrow  B \Rightarrow  \lozenge B \rightsquigarrow \ominus  A^c \tag{S4} \]

\[ C \rightarrow \lozenge A^c\enspace\&\enspace  C \rightarrow  B \Rightarrow  B \rightsquigarrow \lozenge A^c \tag{S5} \]

Note the order of the terms in the conclusion.

We can validate $AE^cI^p$ with

\[ C \rightarrow A\enspace\&\enspace  C \rightarrow \ominus B^c \Rightarrow  A \rightsquigarrow \ominus B \tag{S6} \]

Here the use of the bicontingent seems to be essential in order to apply "transition" (the Ex$\ominus$ rule). We can of course replace the
contingent2 premise with its contingent4 version. $EE^cO^p$ is validated in a similar way to the previous syllogisms.

In the case in which one premise is particular in A21 Aritistotle considers the following valid:

$AI^cI^p, A^cII^c, IA^cI^c, I^cAI^p, EI^cO^p, E^cIO^c, IE^cI^c, O^cAO^p$

$AI^cI^p$ can be validated in the same way as $AA^cI^p$.  $A^cII^c$ cannot be validated with a contingent4 premise:

\[ \ominus C \rightarrow \ominus A\enspace\&\enspace  C \rightsquigarrow  B  \]

But it can with a contingent1, contingent2, contingent3 or contingent5 ($\lozenge\ominus$) premise. The same goes for $IA^cI^c$ and $IE^cI^c$. $I^cAI^p$ and $EI^cO^p$  are validated like $A^cAI^p$ and $EA^cO^p$.

Aristotle validates $O^cAO^p$ by \emph{reductio ad impossibile}. We cannot
of course do this in AML but it is interesting to note that if we take the interpretation of the premises as 
\[ C\rightsquigarrow \lozenge A^c\enspace\&\enspace C\rightarrow B \]

and assume with Aristotle that $B\rightarrow \square A$ then we can prove in AML that $C\rightsquigarrow C^c$. Informally that $B\rightarrow \square A$ is false  means that there is a $B$ such that $(\square A)^c$ holds, that is, $B\rightsquigarrow \lozenge A^c$. But in AML we can prove directly that

\[ C\rightsquigarrow \lozenge A^c\enspace\&\enspace C\rightarrow B \Rightarrow B \rightsquigarrow \lozenge A^c \tag{S7}\]

and analogousy for a contingent2 premise.  We can also validate versions with
contingent3 and contingent4 premises.

Although according to Ross p.365 Aristotle held that there is no conclusion from $I^cE$  we however can validate in AML various syllogisms from such premises (for instance, with a contingent1 premise). For instance $ C\rightsquigarrow \lozenge A\enspace\&\enspace C\rightarrow B^c \Rightarrow \lozenge A \rightsquigarrow \lozenge B^c$.

In A22 Aristotle deals with one contingent and one apodeictic premise. The syllogisms with two universal premises claimed as valid are $A^nA^cI^p, A^cA^nI^c, E^cA^nO^c, E^nA^cO^p,A^nE^cI^p$. 
Using axiom T and lemma 2.5 it is not difficult to reduce these syllogisms to the case of one contingent and one assertoric premise. For instance from S1 and S3 for $A^nA^cI^p$ we validate $A^nA^cI^p$ by

\[ C \rightarrow \square A\enspace\&\enspace \ominus C \rightarrow \ominus B \Rightarrow \lozenge A \rightsquigarrow \ominus B \tag{N1} \]

\[ C \rightarrow \square A\enspace\&\enspace  C \rightarrow \lozenge B \Rightarrow  A \rightsquigarrow \lozenge B \tag{N2} \]

and analogously for $A^cA^nI^c$.  Replacing $B$ with $B^c$ in N1 and using Ex$\ominus$ yields a validation of $A^nE^cI^p$. It is easy to derive

\[ C \rightarrow \square A^c\enspace\&\enspace  C \rightarrow \lozenge B \Rightarrow  \lozenge B \rightsquigarrow \lozenge A^c \tag{N3} \]

which validates $E^nA^cO^p$ with contingent3 conclusion. With a contingent4 or contingent2 premise we obtain a contingent5 $\ominus B \rightarrow \lozenge A^c$ conclusion.
 For $E^c A^n O^c$ we  easily obtain

\[ C\rightarrow \lozenge A^c\enspace\&\enspace C\rightarrow \square B \Rightarrow B \rightsquigarrow \lozenge A^c \tag{N4}\]

and analogously for contingent2, contingent3 and contingent4 premises. For a
contingent4 premise we obtain a contingent5 conclusion

\[ \ominus C\rightarrow \ominus  A^c\enspace\&\enspace C\rightarrow \square B \Rightarrow \lozenge B \rightsquigarrow \ominus A^c \tag{N5}\]

Aristotle considers $A^cE^n$ as not yielding a syllogisms and uses examples to prove this. However in AML we can prove for instance

\[C\rightarrow \lozenge A \enspace\&\enspace C\rightarrow \square B^c \Rightarrow \lozenge A \rightarrow \lozenge B^c\]

 But in Aristotle's  example (sleep-sleeping horse-man, sleep-waking horse-man) the 'contingent' universal does not in fact to correspond to contingent1-6 right  but
to the weak assertoric contingent $\blacklozenge (A\rightarrow B)$ we discussed earlier.  That is, this contingent only expresses that there is a state of affairs in which everything sleeping is a sleeping horse. This is different from stating that in all possible states-of-affairs what is sleeping is possibly a horse or what can sleep is possibly horse, etc.
 With this reading we can confirm that $\blacklozenge (C\rightarrow A)\enspace \&\enspace C\rightarrow \square B^c$ does not yield a syllogism in AML using Aristotle's own example. 

When one premise is particular in A22 Aristotle considers as valid:

$A^nI^cI^p, I^cA^nI^p, A^cI^nI^c, I^nA^cI^c,E^cI^nO^c, O^cA^nO^p, E^nI^cO, O^nA^cO$.

It is clear that if in a proof a syllogism in AML we apply weakening to a premise
then we obtain a syllogism also if we start with the weakened form of that premise. Thus $A^nI^cI^p$, $I^cA^nI^p$ and $O^cA^nO^p$ follow directly from $A^nA^cI^p$, $A^cA^nI^c$. and $E^cA^nO^c$. But what about when the apodeictic premise is weakened as in $A^cI^nI^c, I^nA^cI^c$ and $E^cI^nO^c$ ? For $I^nA^cI^c$ we can derive

\[ C \rightsquigarrow \square A\enspace\&\enspace \lozenge C \rightarrow \ominus B \Rightarrow \lozenge A \rightsquigarrow \ominus B \tag{N6} \]

\[ C \rightsquigarrow  \square A\enspace\&\enspace  C \rightarrow \lozenge B \Rightarrow  A \rightsquigarrow \lozenge B \tag{N7} \]

Note that in N6 the premise must now be contingent5. The syllogisms

$A^cI^nI^c$ and $E^cI^nO^c$ are validated in a similar way.

Consider now the surprising syllogisms $E^nI^cO, O^nA^cO$ which yield
assertoric conclusions. In $AML$ where we can in fact derive

\[C \rightarrow \square A^c\enspace\&\enspace C\rightsquigarrow \lozenge B \Rightarrow \lozenge B \rightsquigarrow A^c \]

that is, from a contingent1 premise we get a contingent7 conclusion.

In $AML^{S5}$ where we can show that $\lozenge \square A \rightarrow \square A$ we have
\[C \rightarrow \square A^c\enspace\&\enspace \lozenge C\rightsquigarrow \lozenge B \Rightarrow \lozenge B \rightsquigarrow A^c \]

and for a contingent4 premise we obtain a contingent8 conclusion.
For $O^nA^cO$ we have  in AML  

\[C \rightsquigarrow \square A^c\enspace\&\enspace C\rightarrow \lozenge B \Rightarrow \lozenge B \rightsquigarrow A^c \]

and similarly for contingent2 and contingent3 premises.  It does not seem that we obtain a valid syllogism for a contingent4 premise but we do for a contingent5 one. The situation for $I^cE^n$ is the same as for the discussion of $A^cE^n$.
This can be shown not to yield a syllogism if we interpret contingent as weak assertoric.

\section{Conclusion}

We were faced with the task of finding an axiomatic-deductive system (preferably with a sound semantics) which meets the following desiderata: \\

i) It should be homogenous in nature with Aristotle's logic. That is, have a similar syntactic style ( distinct from quantifier logic and closer to natural language) that could be easily grasped by Aristotle himself.

ii) The axiomatic-deductive system should economical and elegant and be closer to Aristotle's own proofs than to modern first-order logic.

iii) It should be more flexible and refined than Aristotle's logic in the sense that we can embed Aristotle's system  into it and resolve  ambiguities and seeming inconsistencies by having at our disposal a more refined spectrum of different notions of contingency, apodeictic and assertoric predication to choose from in different contexts. It should allow us look at the problems under a microscope so as to be able to pin-point the problems and easily rectify or slightly and naturally  modify it into a completely consistent system

iv) It should still  have a connection  to modern mathematical logic, specially modern S5 modal propositional and quantifier logic as well as algebraic logic.  The deductive system should ideally be related to natural deduction\footnote{Note that we can see natural deduction quantifier rules at work for instance in Aristotle's proof in 24a14-20} and have a sound semantics similar (but not necessarily identical) to Kripke frames.

v) It should (if possible) furnish us with an interpretation  which validates all the syllogisms  in A8-22 and it should not derive conclusions from pairs of premises which Aristotle considers inconclusive.

vi) It not only validates but extends Aristotle's system and is an interesting system in its own right.\\

The systems of McCall \cite{call} and Malink\cite{malink} build upon the pioneering work of Bocheński\cite {bo}, Becker\cite{beck} and  Łukasiewicz\cite{luka}. Malink's system is one of the most interesting to date as a sound first-order semantics (called\emph{predicable semantics})is also considered which establishes as inconcludent all the syllogisms held to be inconcludent by Aristotle. 
Both McCall's and Malink's system do not meet desiderata i) and ii).   McCall's system is syntactically close to Aristotle but is based on a very  large number of syllogisms  which must be take as axioms and also a fairly large
number of deduction rules.  Malink's system
is based on full  first-order quantifier logic (without equality)  and six axioms involving three primitive binary predicates $\mathbf{A},\mathbf{N}$ and $\mathbf{\hat{N}}$.
In order to express the model syllogisms Malink must build up nine complex first-order notions from these predicates (including
introducing "ontological" distinctions between elements of the model such as the by means of the unary predicate $\mathbf{\Sigma}$). The proofs of the validity of the moods of the modal syllogistic given
in Appendix B of \cite{malink}  are often quite complex.  
Thus it is arguable that AML is much closer to meeting desiderata  i) and ii) than Malink's or McCall's systems.
Note that AML could also be recast in first-order form with three sorts (individuals, concepts, states-of-affairs) and a single primitive trinary predicate $\Delta(x,t,v)$.
Then we would define different modes of predications in AML such as $a\rightarrow \lozenge b$ as a predicate $Q(a,b) \equiv \forall v. \forall x. \Delta(x,a,v) \rightarrow \exists v'. \Delta(x,b,v')$. Then
we get would by first-order logic alone that $Q(a,b) \enspace\&\enspace Q(b,c) \rightarrow Q(a,c)$. 
Note also that AML is quite distinct  (both proof-theoretically and semantically) from early approaches such as Becker's which employ directly modern first-order modal logic and that there is no obvious way to embed AML into it.  AML is only analogous to modern S5 modal logic.

Malink's system does meet iii) . For instance, it is based on two kinds of necessitation just as AML we can have both ordinary $\lozenge$ and strong contingency $\ominus$.  However
Malink's system is based on binary predicates so has a more $\emph{de dicto}$ flavour while the $\emph{de re}$ approach of AML allows us to consider using $\lozenge,\ominus,\square,\sqcup$ (and no symbol at all) up to  25 types of predication between two terms.  For iii) with AML we draw the following conclusions:\\

a) Although it is clear that Aristotle distinguished clearly two types of contingency at least six  (if not eight) inter-related types of contingency have to be considered in order to obtain a clear and fully consistent theory of the modal syllogism. AML also greatly clarifies the problem of the conversion of particular contingents. Also AML$^{S5}$ proves the conversion of the negative necessary universals.

b) We find  in places an ambiguity and confusion between weak and strong assertoric predication and between necessary and strong assertoric predication \cite{loeb}[p.190]) as well as between contingent and weak assertoric predication.  This agrees with Becker's\cite{beck} take on the classical problem of the oscillation between \emph{de dicto} and \emph{de re}.\\

Point iv) is clearly met by AML.  The analogues of modern modal axioms appear as essential to the  structure of AML proofs of Aristotle's theory of the modal syllogism. S5 makes a surprising appearance in a few places. Aristotle seems to have explicitly acknowledged K.  Neither McCall's or Malink's  system seem to have much connection to modern modal logic although
ax$_6$\cite{malink}[p.287] might be read as an analogue of  T.

 As for v)  McCall's only validates a certain limited  amount of Aristotle's syllogisms. Malink's system on the hand
 not only validates Aristotle's conclusions but even furnishes refutations of all the moods held to be inconcludent. There is however still the question of
 the legitimacy of the interpretations given to the various form of predication, as for instance the interpretation of the negative necessary predication  $\mathbb{N}^e a  b$  seems rather forced since it introduces a disjunction $\hat{\mathbf{\Sigma}} a \vee \hat{\mathbf{\Sigma}} b$ in the definition. Also Malink's six axioms can be seen as solving the "problem of the two Barbaras" by positing these two Barbara moods \emph{ad hoc} as axioms.  But it is generally considered that  axioms should express intuitively self-evident truths or propositions generally agreed upon.
 
 AML (or more precisely AML$^{S5}$ for a few
syllogisms) allows us to give interpretations to "contingent", "problematic", "assertoric" and "apodeictic" which validate all the syllogisms in A8-22.  AML however proves conclusions for a fair number of combinations of premises considered inconclusive by Aristotle (though agreeing with Aristotle about the inconclusivity of some non-obvious combinations). In AML we can explain
these discrepancies by examining the counter-examples Aristotle offers. We note that Malink's system also has difficulties agreeing with Aristotle's concrete counter-examples. According to the AML analysis in many of Aristotle's examples  we find an ambiguity or confusion between weak and strong assertoric predication and between necessary and strong assertoric predication as well as between contingent and weak assertoric predication.  

Finally it seems that both AML  and Malink's system fulfilll desideratum vi). 
Also we observe that while AML is connected to extensional models (based on sets of individuals)  Malink's predicative semantics is intensional in flavour (indeed heavily inspired by the ontology and logic of the \emph{Topics}) and based on quantification over terms rather than individuals (cf. \cite{rini}[p.236]), that is, we are arguably in the presence of disguised second-order logic.  Whereas AML required the $\star$ operation to deal with empty extensions and the passage from A to I predication, Malink solves this problem by the first axiom $\mathbf{A}aa$. Obviously
we cannot read this to say: unicorns are unicorns so the term unicorn
has non-empty extension.

There remain the following open problems regarding AML and its extensions:\\

1.  Obtain a complete lists of all valid syllogisms  in AML and its extensions (including all types of predication) perhaps employing an automatic theorem prover. 

2. Investigate the complexity of the simplest model required to refute a non-valid syllogism.

3. Prove the independence of the axiom S5 in AML and find the right extension of AML that can be used to prove completeness for extensional models.

4. Investigate more complex types of models for AML with an accesibility relation  between states-of-affairs. This seems a good strategy to prove the independence of S5. Find such a model in which S5 fails.

5. Investigate modal sorites and analyse the rest of the theory of the Prior and Posterior Analytics.

6. Extend the logic of AML to include the intuitionistic and classical negation rules in order to formalise Aristotle's \emph{reductio ad impossibile} arguments
and his theory of contraries and oppositions. The meta-logic of AML is in fact weaker than the meta-logic used by Aristotle
but we saw an example of Aristotle's use of the classical negation rule  suggesting AML could easily be extended with negation to consistently capture  a larger fragment of Aristotle's meta-logic

7.  A further topic would be to explore the alternative presentation of AML  in
which a binary term operation $\wedge$ is introduced. In this case the semantics looks like a collection of \emph{representations} of a Boolean algebra with operators (in this case a single operator $\square$). 

8. Another approach which comes to mind is a purely algebraic treatment of  AML,  to consider AML as a kind of generalised  lattice with operators, independently of its semantics (representation) in terms of the set-theoretic class of models defined. This might be seen as taking an intensional approach in which we focus on the relationship between concepts.

9. If we drop rule K from AML then we can derive $A^nA^nA^n$, $A^nAA^n$ but not $AA^nA^n$ in the first figure (problem of the two Barbaras). If we extend AML to include disjunction then we might devise
other interpretations of contingency such as $A\rightarrow \lozenge B \enspace\vee\enspace \lozenge A \rightarrow B$ so that many other moods remain valid.  We could thus  investigate weaker versions of AML for modelling the modal syllogistic and attempt to find a suitable semantics for which rule K is no longer sound.

10. Extend this approach to a calculus of relations. Slomkowski has argued in \cite{slo} that the \emph{Topics} includes the rudiments of a
calculus of relations which is given fuller treatment in Galen's \emph{Introduction to Logic} (see the article by Barnes in \cite{Sharp}). The approach in this paper might be extended to binary predicates (relations). Let $\Delta$ be the diagonal relation. Then unary predicates are represented by terms of the form $R\wedge \Delta$. We use $\rightarrow$ and $\leadsto$ as before. We have furthermore a product operation on terms $R\circ S$ and an inversion operator $R^{*}$. We could thus obtain a straightforward extension of AML in which could express the analogue of more complex sentences in first-order logic but with the added benefit of modalities. Algebraically this gives rise to structure similar to that of a \emph{quantale}\cite{Q} but with an operator. 

\section*{Acknowledgments}

The author would like to thank the anonymous referees for many helpful comments and criticisms which greatly contributed to improving the contents and presentation of this paper.

								\end{document}